\documentclass[12pt,reqno]{amsart}

\usepackage{amsmath, amsthm}

\usepackage{geometry}
\usepackage{fullpage}
\usepackage[all]{xy}

\usepackage[mathcal,mathscr]{euscript}
\usepackage[OT2,T1]{fontenc}
\usepackage[english]{babel}
\usepackage{newtxtext}
\usepackage{newtxmath}

\usepackage{etoolbox}
\patchcmd{\section}{\scshape}{\bfseries}{}{}
\makeatletter
\renewcommand{\@secnumfont}{\bfseries}
\makeatother
\numberwithin{equation}{section}
\usepackage{enumitem}

\newtheorem*{introtheorem}{Main Theorem}
\theoremstyle{definition}
\newtheorem*{introremark*}{Remark}
\theoremstyle{plain}
\newtheorem{theorem}{Theorem}[subsection]
\newtheorem{proposition}[theorem]{Proposition}
\newtheorem{lemma}[theorem]{Lemma}
\newtheorem{corollary}[theorem]{Corollary}
\theoremstyle{definition}
\newtheorem{remark}[theorem]{Remark}
\newtheorem{example}[theorem]{Example}

\NewCommandCopy{\proofqedsymbol}{\qedsymbol}
\AtBeginEnvironment{proof}{\renewcommand{\qedsymbol}{\proofqedsymbol}}

\AtBeginEnvironment{example}{%
  \pushQED{\qed}\renewcommand{\qedsymbol}{$\lozenge$}%
}
\AtEndEnvironment{example}{\popQED}

\AtBeginEnvironment{remark}{%
  \pushQED{\qed}\renewcommand{\qedsymbol}{$\lozenge$}%
}
\AtEndEnvironment{remark}{\popQED}

\usepackage{mathtools}

\usepackage[font=small, labelfont=bf]{caption}

\newcommand{\Z}{\mathbf{Z}}
\newcommand{\R}{\mathbf{R}}
\newcommand{\Q}{\mathbf{Q}}
\newcommand{\Cc}{\mathbf{C}}
\newcommand{\F}{\mathcal{F}}
\newcommand{\Hh}{\mathbb{H}}

\newcommand{\DD}{\mathbb{D}}
\newcommand{\B}{\mathbb{B}}
\renewcommand{\geq}{\geqslant}
\renewcommand{\leq}{\leqslant}
\renewcommand{\Re}{\mathrm{Re}}

\renewcommand\d{\mathop{}\!\mathrm{d}} 
\DeclareMathOperator*{\LH}{\mathrm{Lim}} 

\makeatletter
\newcommand*{\defeq}{\mathrel{\rlap{%
                     \raisebox{0.24ex}{$\m@th\cdot$}}%
                     \raisebox{-0.24ex}{$\m@th\cdot$}}%
                     =}
\makeatother

\begin{document}
\date{\today}
\title[Eisenstein series measure]{The limiting zero counting measure for Eisenstein series    } 
\author[G.~Cornelissen]{Gunther Cornelissen}
\author[M.~Mehrabdollahei]{Mahya Mehrabdollahei}
\address{\normalfont Mathematisch Instituut, Universiteit Utrecht, Postbus 80.010, 3508 TA Utrecht, Nederland}
\email{g.cornelissen@uu.nl}
\address{\normalfont Max-Planck-Institut f\"ur Mathematik, Postfach 7280, 53072 Bonn, Deutschland}
\email{mehrabdollahei@mpim-bonn.mpg.de}

\thanks{The authors thank the Max-Planck-Institut f\"ur Mathematik in Bonn as well as Universiteit Utrecht for excellent conditions during work on this paper.  This work was partially supported by the Dutch Research Council (NWO) through the grant number OCENW.M20.233.\\ \indent \emph{Statement on the use of AI:} Anthropic's Claude Pro (Opus 5.5) was used to explore the material for one specific rational function instead of modular forms; all further development for modular forms, including proofs and the writing of the paper was done by the authors. OpenAI’s Sol (GPT 6.1 Medium) was used to audit the paper after completion, leading to several more edits in the final version.}  

\begin{abstract} \noindent For an arbitrary congruence group $\Gamma$, we study the zeros of a generating set of the space of Eisenstein series of varying even weight for all $\Gamma(1)$-conjugates of $\Gamma$. These zeros form a $\Gamma(1)$-invariant set that descends to the modular orbifold $Y(1)$. It is known that, as the (even) weight tends to infinity, the set of zeros converges in the Hausdorff metric to finitely many geodesic segments. 

In this paper, we find the weak  limit of the normalised counting measure of the zeros in the (natural) orbifold sense: if a geodesic occurring in the Hausdorff limit is parametrised by the unit circle $S^1$ via a M\"obius transformation then the limit measure on this segment is the pullback of the Haar measure on $S^1$ by that transformation. For the proof, we first show a version of the Poincar\'e--Lelong theorem for orbifold curves, and then apply it on $Y(1)$ to reduce the problem to computing the scaled limit of the subharmonic functions $\log |E|$ associated to Eisenstein series $E$. This turns out to be the maximum of the negative logarithm of the modulus of finitely many linear functions supported on the relevant geodesic segments, and we finish the proof by showing that the associated measure has no atoms and making explicit, on each open segment, the distributional Laplacian of this function in terms of the argument measure. 

The non-atomicity allows us to conclude a sublinear growth rate in the weight for the total order of all Eisenstein series at a given point in $Y(1)$, and weak convergence also of the non-orbifold measure to the same limit.  
\end{abstract}

\maketitle

\section{Introduction}

The study of zeros of Eisenstein series goes back at least to the last joint paper of Hardy and Ramanujan from 1918 \cite{HardyRamanujan}. For classical Eisenstein series on $\mathrm{PSL}_2(\Z)$, the 1970 work of F.~Rankin and Swinnerton--Dyer \cite{RSD} provided an equally simple as definitive approximation argument, implying that all zeros lie on a specific geodesic, and equidistribute as the weight tends to infinity. For more history on the study of zeros of Eisenstein series and more general modular forms, as well as pictures of numerical simulations and precise calculations of the zero sets and their limits for various congruence groups, we refer to \cite{CCR}. 

Here, we focus on new limit results from the measure theoretic point of view. A fundamental result is due to Rudnick \cite{Rudnick}, who computed the limit measure for zeros of cuspidal Hecke eigenforms; their zeros equidistribute in the entire upper half plane. Although a marked difference between cuspidal forms and Eisenstein series lies in the fact that the limit measure for general Eisenstein series is supported on a finite union of geodesic segments (so a concentration phenomenon), some of the techniques from \cite{Rudnick}, notably potential theory in the complex plane, will also feature prominently in our work. This should come as no surprise, as the method features prominently in the study of zeros of truncated power series (originating in the work of Szeg\H{o} on his eponymous curve \cite{Szego}), zeros of orthogonal polynomials, etc., see, e.g., \cite{ST}. 

\subsection*{Setup} We fix a congruence group $\Gamma \leq \Gamma(1)=\mathrm{PSL}_2(\Z)$ of level $N$ (meaning $\Gamma(N)$ is the maximal principal congruence subgroup contained in $\Gamma$) with normaliser $N_{\Gamma(1)}(\Gamma)$ in $\Gamma(1)$ and consider the \emph{Eisenstein series of even weight $k \geq 4$ and label $(a,b) \in (N^{-1}\Z/\Z)^2$} defined by 
\begin{equation} E^{a,b}_{k,\Gamma}=E^{(a,b)}_{k,\Gamma} \defeq \hspace*{-4mm}  \sum_{g \in \Gamma(N) \backslash \Gamma} \hspace*{-3mm}  E_k^{(a,b)g} \mbox{ with } E_k^{(a,b)}(z) \defeq \hspace*{-5mm}  \sum_{(0,0) \neq (m,n) \in \Z^2}  \hspace*{-2mm} \frac{\exp(2 \pi i (an-bm))}{(mz+n)^k}, \end{equation} 
where $(a,b)g$ is given by the right action of a lift of the matrix $g\in\Gamma(1)$ to $\mathrm{SL}_2(\Z)$ on the row vector $(a,b)$; since the weight is even, this is independent of the choice of lift. For a discussion of other typical choices of generators for the Eisenstein part of the space of modular forms for $\Gamma$, see \cite[\S 2]{CCR}.
 
Let \begin{equation} N_\Gamma \defeq \mathrm{N}_{\Gamma(1)}(\Gamma) \backslash \Gamma(1) \mbox{ and } c_\Gamma \defeq N^2 \cdot |N_\Gamma|.\end{equation} For even weight $k \geq 4$, we define a \emph{norm modular form}  
\[ \mathcal N_{k,\Gamma}\defeq \hspace*{-5mm} \prod_{\substack{(a,b) \in (N^{-1}\Z/\Z)^2 \\ \gamma \in N_\Gamma}} \hspace*{-5mm}  E_{k,\gamma^{-1} \Gamma \gamma}^{a,b}.\] This is  a non-zero modular form of weight $c_\Gamma k $ for $\Gamma(1)$, not vanishing at the cusp. We consider its $\Q$-divisor 
\[ Z_{k,\Gamma} \defeq  \mathrm{div}(\mathcal N_{k,\Gamma}) = \sum_{x \in Y(1)} \frac{\mathrm{ord}_x(\mathcal N_{k,\Gamma})}{\# \mathrm{Stab}_{\Gamma(1)}(x)} x \]  
on the orbifold $Y(1) = [\Gamma(1) \backslash \Hh]$, with $\Hh$ the hyperbolic complex upper half plane; the only orbifold points are the orbits of 3\textsuperscript{rd} and 4\textsuperscript{th} roots of unity in $\Hh$. Equivalently, we consider this divisor in the standard fundamental domain $\F$ for the action of $\Gamma(1)$ on $\Hh$, given as 
\[ \F \defeq \{ z \in \Hh \colon |z|>1 \mbox{ and }-1/2 \leq \Re(z) < 1/2 \} \cup \mathcal A \mbox{ with } \mathcal A \defeq \{ e^{i \theta} \mbox{ with } \pi/2 \leq \theta \leq 2 \pi/3 \}.\]
By the valence formula and non-vanishing at the cusp, this $\Q$-divisor has degree 
\begin{equation} \label{valence} \deg(Z_{k,\Gamma}) = \sum_{x \in Y(1)} \frac{\mathrm{ord}_x(\mathcal N_{k,\Gamma})}{\# \mathrm{Stab}_{\Gamma(1)}(x)} = \frac{c_\Gamma k}{12}.\end{equation}

\subsection*{Main results} In this paper, we are interested in the corresponding natural \emph{orbifold normalised zero-counting measure} \[ \mu_{k,\Gamma} \defeq  \frac{1}{\deg(Z_{k,\Gamma})} \sum_{x \in Y(1)} \frac{\mathrm{ord}_x(\mathcal N_{k,\Gamma})}{\# \mathrm{Stab}_{\Gamma(1)}(x)} \delta_x \] on $Y(1)$, where $\delta_x$ is a Dirac point mass at $x \in Y(1)$; more precisely,  we are interested in the weak  limit of $\mu_{k,\Gamma}$ as $k \rightarrow + \infty$ over all even $k \geq 4$ for  a general congruence group, answering Question (vii) from \cite[\S~9]{CCR}. To formulate the result, we need the following concepts.

Write $|Z_{k,\Gamma}|$ for the support of $Z_{k,\Gamma}$, seen as a plain (unweighted) set. In \cite[Thm.~E]{CCR}, it was shown that $|Z_{k,\Gamma}|$ converges to a union of finitely many geodesic segments in the metric on $Y(1)$ inherited from the Hausdorff metric on the closure $\overline \F$. Every segment in the Hausdorff limit of $|Z_{k,\Gamma}|$ is itself part of the Hausdorff limit of the set of (not necessarily $\Gamma(1)$-invariant) zeros of $E_{k,\gamma^{-1}\Gamma\gamma}^{(a,b)}$ for some ---possibly multiple--- choices of $a,b,\gamma$. Each segment in the limiting zero set is contained in a geodesic
$
\{z\in\mathbb H\mid |\sigma\cdot z|=1\}
$
for a very specific matrix \(\sigma\in \mathrm{M}_2(\Z)\cap\mathrm{GL}_2(\Q)\); see \cite[Theorem E]{CCR} and Subsection \ref{geodesics} on how to find $\sigma$ in terms of the Eisenstein series. We denote the corresponding geodesic segment in $Y(1)$  by \(E_\sigma\). It is the image of a connected segment $I_\sigma$ of the unit circle under the action of the Möbius transformation \(\sigma^{-1}\). We index these $\sigma$ by a finite multiset $\mathcal I_\Gamma$ in which \(\sigma\) occurs with multiplicity equal to the number of distinct
$
(a,b,\gamma)\in N^{-1}\Z/\Z\times N^{-1}\Z/\Z\times N_\Gamma
$
for which the segment $E_\sigma$ occurs in the total limit. Then we have the following.

\begin{introtheorem} \label{main} 
For any fixed congruence group $\Gamma$ and even weight $k$ tending to $+\infty$, the normalised orbifold counting measure of the norm modular form converges weakly 
\begin{equation} \label{mainconvform} \mu_{k, \Gamma} \rightarrow  \mu_{\infty,\Gamma} \defeq  \frac{6}{c_\Gamma \pi} \sum_{\sigma \in \mathcal I_\Gamma} |\d\arg(\sigma \cdot z)|_{E_\sigma} \ = \frac{12}{c_\Gamma} \sum_{\sigma \in \mathcal I_\Gamma} \sigma^* \left(\mu_{S^1}|_{I_\sigma} \right)\end{equation} 
where $\mu_{S^1}$ is the normalised Haar measure on $S^1$, and $\mathcal I_\Gamma$ is the multiset defined above. 
\end{introtheorem} 
 
For technical reasons, we need to consider the zeros in an open set containing a compact set enveloping the zeros in $\overline \F$, but we will ignore this in the following brief outline of the proof. First, we fix $a,b,\gamma$ and use the fact from \cite{CCR} that the zeros of $F_k = E_{k,\gamma^{-1} \Gamma \gamma^{-1}}^{(a,b)}$ belong to a compact subset of $\overline \F$ and $H_k(z) \defeq \frac{1}{k} \log |F_k(z)|$ converges on that set minus the Hausdorff limit to an explicit function $\Phi(z)$ in $L^1_{\text{loc}}$, given as the maximum of minus the logarithm of the modulus of a finite set of linear forms, see Section \ref{notaux}. The compactness theorem for subharmonic functions implies convergence on the entire domain. Then the orbifold zero counting measure for $F_k$ is the Laplacian of the harmonic function $H_k(z)/2\pi$ in the sense of distributions; to show this, we prove a version of the Poincar\'e--Lelong theorem in the orbifold curve setting that may be of independent interest, see Theorem \ref{pol}. Standard properties of (sub-)harmonic functions imply convergence of the zero counting measures to $\Delta \Phi / 2\pi$. Globally, these can be added up and normalised over all choices of $a,b,\gamma$, thanks to the valency formula, see Section \ref{ptmc} and Proposition \ref{propPhi}. By suitable descent to $Y(1)$, we find vague convergence of measures, and tightness implies weak convergence. Finally, using the jump formula for distributional Laplacian on every occurring segment, the limit distribution is computed in terms of the argument, connecting it to the pullback of the Haar measure, see Section \ref{lmhaar}. We make everything explicit for principal congruence groups, see Examples \ref{gamman}--\ref{gamman4}.  In Subsection \ref{noatoms} we show that the limit measure has no atoms, which we use in Subsection \ref{finalsec} to prove that 
\begin{enumerate} 
\item $\mathrm{ord}_x (\mathcal N_{k,\Gamma})$ is upper bounded by $o(k)$ uniformly over $x \in Y(1)$ (Corollary \ref{smallo}); 
\item the `non-orbifold' counting measures 
\[ \nu_{k,\Gamma} \defeq \frac{\displaystyle{\sum\nolimits_{x \in Y(1)} \mathrm{ord}_x(\mathcal N_{k,\Gamma}) \delta_x}}{\displaystyle{\sum\nolimits_{x \in Y(1)} \mathrm{ord}_x(\mathcal N_{k,\Gamma})}} \] converge weakly to the same limit measure $\mu_{\infty,\Gamma}$ on $Y(1)$ (Corollary \ref{noorb}).
\end{enumerate}  

 \begin{introremark*} In \cite[Prop.~H]{CCR}, it was shown that the arguments of the zeros of $\mathcal N_{k,\Gamma(N)}$ in $\F_1 \defeq \F \cap \{ \mathrm{Re}(z) \leq 0\}$ equidistribute in $\mathcal A$ as $k \rightarrow + \infty$; passing to $Y(1)$, this implies the main theorem in that special case. The proof of angular equidistribution in the special case was based on precise results for the location of `most' zeros using Rouch\'e's theorem, and also gave details on the angular convergence rate: it is of order $O(1/k)$ in general, with a proportion of zeros lying \emph{on} the arc for finite weight, and some exceptional zeros converging at exponential speed to the arc. 
\end{introremark*} 

\begin{introremark*} Eisenstein series on $\F$ are well-approximated by rational functions given as a suitable truncation of their defining sum, say, for example, $f_k(z) \defeq (z(z+1))^k+c_1 (z+1)^k + c_2z^k$ with $c_1c_2\neq 0$ on $\Gamma(3)$. The limit in $k$ of the zero set of $f_k$ in the entire complex plane can be mapped conformally to the unit disk, and then Haar measure can be pulled back via that map to the limit set and restricted to the intersection with $\F_1$. However, this does \emph{not} equal the corresponding limit measure of the zeros of the Eisenstein series. 
\end{introremark*} 

\section{The orbifold Poincar\'e--Lelong theorem} 

The classical Poincar\'e--Lelong theorem identifies the zero measure of a holomorphic function $f$ in some domain $D \subseteq \Cc$ with a constant multiple of the Laplacian of $\log |f| \in L^1_{\text{loc}}(D)$, seen as distributions, see e.g.\ \cite[Thm.~3.7.8]{Rans}. We now prove and state a version for sections of Hermitian line bundles on orbifold curves. This may be of independent interest, albeit surely known to the experts.  

We then apply it to the modular orbifold $[\Gamma(1) \backslash \Hh]$ (a global `stacky' quotient), with modular forms of even weight seen as sections of powers of the canonical line bundle, given the Petersson metric locally. We refer to \cite{PM} and \cite{Hain} for details on the constructions surrounding the modular orbifold (in their terminology, we are considering the `effective' orbifold $[\Gamma(1) \backslash \Hh]$ instead of $[\mathrm{SL}_2(\Z) \backslash \Hh]$, the latter object being the moduli space of elliptic curves in which, however, every point has a non-trivial stabiliser), and to \cite{Pipi} for the theory of modular forms in this context. 
 
\subsection{Orbifolds, orbifold bundles and $\Q$-divisors} We recall the definitions and refer to \cite{Coman} for, in particular, the spelling out of somewhat unwieldy but natural gluing conditions. A \emph{complex orbifold chart}  at point $p$ is a covering $\pi \colon \widetilde D \rightarrow D$ of a complex disk $D$ in the local coordinate $z$ at $p$, such that, choosing a local coordinate $w$ on $\widetilde D$, the covering map is the quotient for an action of a generator of the cyclic group $\mu_{m}$ given by $w \mapsto \zeta_{m}w$, $\zeta_{m} = e^{2 \pi i/m}$. Hence $\pi$ is defined locally by $z=w^m$, of degree $m$ and, for $m>1$, totally ramified above zero. An \emph{(effective) orbifold curve} is such a connected space glued from orbifold charts such that all but finitely many points have $m=1$; points where $m>1$ are called \emph{orbifold points}.  

To transfer a standard geometric concept from the manifold to the orbifold setting, the principle is to define it in the usual way using the disk $\widetilde D$ and require equivariance for the $\mu_m$-action. 
Thus, an \emph{orbifold line bundle} $\mathcal L$ is a $\mu_m$-equivariant holomorphic line bundle over each orbifold chart disk $\widetilde D$, glued appropriately. A \emph{meromorphic section} $s$ of $\mathcal L$ around $p$ is given by a meromorphic $\mu_m$-equivariant section $\tilde s$ of $\mathcal L$ on $\widetilde D$, i.e., such that $\tilde s (\zeta_m w) = \zeta_m \cdot \tilde s (w)$. Choosing a frame (non-vanishing local trivialisation) $\kappa \colon \widetilde D \rightarrow \mathcal L$ that is equivariant, i.e., such that $\zeta_m \cdot \kappa(w) = \zeta_m^a \kappa(\zeta_m w)$ for some \emph{isotropy weight} $a \in \Z/m\Z$, we can write every meromorphic section as $\tilde s = \tilde f \cdot \kappa$ for some scalar function $\tilde f \colon \widetilde D \rightarrow \Cc$, and the equivariance is expressed by the condition that $\tilde f ( \zeta_m w) = \zeta_m^a \tilde f(w)$. It follows that $\mathrm{ord}_0 (\tilde f) = a \mbox{ mod } m$. The \emph{(orbifold) order} of $s$ at $p$ is \[\mathrm{ord}_{p}(s) \defeq \mathrm{ord}_0(\tilde f)/m\] (the $\mathrm{ord}_0(\tilde f)$ on the right is the usual order of the function $\tilde f$ at $0$), and the $\Q$-divisor on $X$ associated to $s$ is \[\mathrm{div}(s) = \sum_{p \in X} \mathrm{ord}_p(s) \cdot p \in \prod_{p \in X} \Q \cdot p,\]
where the divisors are locally finite (on a non-compact curve, infinitely many isolated zeros can occur, so the direct product cannot be replaced by a direct sum in general). By replacing $p$ by $\delta_p$ in the definition, the divisor can be seen as a current. By slight abuse of notation, we will also denote this current as $\mathrm{div}(s)$. 

\subsection{Metrics and currents on orbifolds} We fix a Hermitian metric $\tilde h$ on an orbifold line bundle $\mathcal L$ and write $||\cdot||_h$ for the associated smoothly varying norm on the fibers, requiring it to be $\mu_{m}$-invariant on each orbifold chart $\widetilde D$. The invariance implies that there is a unique function, denoted $||s||_h$, on $D$ such that $\pi^* ||s||_h = ||\tilde s||_h$. Fixing an equivariant frame $\kappa$ on $\widetilde D$, define the \emph{local weight} to be the smooth real function $\phi \defeq - \log ||\kappa||_h^2$. Furthermore, if we change the frame by $\kappa \rightarrow \kappa g$ for a nowhere vanishing holomorphic function, the local weight changes by $\phi \rightarrow \phi - 2 \log |g|$. Let $\Delta_w = 4 \partial_w \partial _{\bar w} = \partial_x^2 + \partial_y^2$ be the Laplace operator in the complex plane. Since $g$ is holomorphic, $\log |g|$ is harmonic, and hence $\Delta \phi$ does not depend on the choice of frame. From compatibility with the group action, we find that $\phi \colon \widetilde D \rightarrow \R$ is $\mu_m$-invariant, so there is a well-defined current on $X$, the \emph{Chern form} $c_1(\mathcal L,h)$ determined by $\pi^* c_1(\mathcal L,h) \defeq \frac{1}{4\pi} \Delta_w \phi$. 

\begin{theorem}[Orbifold curve Poincar\'e--Lelong theorem] \label{pol} Let $s$ denote a meromorphic, not identically zero, section of an orbifold line bundle  $\mathcal L$ with smooth equivariant Hermitian metric $h$ on an orbifold curve $X$; then, as currents on $X$,
\[ \frac{\Delta}{2 \pi} \log ||s||_h = \mathrm{div}(s) - c_1(\mathcal L,h). \]  
\end{theorem}

\begin{proof} 
This is a local statement. In a sufficiently small neighbourhood of a point $p$ of $X$ with orbifold chart $\pi \colon \widetilde D \rightarrow D$ of degree $m$ (possibly $m=1$, local uniformisers $w$ and $z$, as before), trivialise $\mathcal L$ using a frame $\kappa$ and write $\tilde s=\tilde f \cdot \kappa$ for a meromorphic function $\tilde f$ on $\widetilde D$. Choose the neighbourhood so small that $\tilde f = w^r g(w)$ for some integer $r$ and an invertible function $g$. We then have 
\[\log ||\tilde s||_h = \log | \tilde f| + \log ||\kappa||_h = \log | \tilde f| - \phi/2,\] and hence, plugging in the definitions,  
\begin{equation} \label{deltaw} \frac{\Delta_w}{2\pi} \pi^* \log ||s||_h = \frac{\Delta_w}{2\pi} \log |\tilde f| - \pi^* c_1(\mathcal L,h).\end{equation} 
We note that, if $u \in L^1_{\text{loc}}(D)$ and $\pi^* u  \in L^1_{\text{loc}}(\widetilde D)$ then, as distributions, \begin{equation} \label{deltadelta} \pi_* \Delta_w \pi^* u = m \Delta_z u.\end{equation} 
Indeed, integrating against a test function $\varphi \in C_c^\infty(D)$, the chain rule for $\varphi \circ \pi$ and a change of variables gives 
\[ \langle \pi_* \Delta_w (\pi^* u), \varphi \rangle = \int_{\widetilde D} (u \circ \pi) |\pi'|^2 (\Delta_z \varphi)\circ \pi = m \langle \Delta_z u, \varphi \rangle. \]
We apply the identity \eqref{deltadelta} to $u = \log ||s||_h/2\pi$; the integrability conditions for $u$ and $\pi^*u$ are satisfied since locally, $\pi^* u = r/2\pi \log |w| + b(w)$ with $b$ a smooth function.  
In this way, we find 
\[ \frac{\Delta_z}{2\pi} \log ||s|| = \pi_* \frac{\Delta_w}{2m \pi} \log |\tilde f| - c_1(\mathcal L,h).\]
The classical Poincar\'e--Lelong formula (\cite[Thm.~3.7.8]{Rans}) applied to $\tilde f$ on $\widetilde D$ gives $\Delta_w \log |\tilde f| / 2\pi = r \delta_0 = \mathrm{ord}_0(\tilde f) \delta_0$. Thus, with $\pi(0)=p$\ and $\mathrm{ord}_p(s)$ defined as $\mathrm{ord}_0(\tilde f)/m$, we find 
 \[ \frac{\Delta_z}{2\pi} \log ||s|| = \mathrm{ord}_p(s) \delta_p - c_1(\mathcal L,h)\]
 in a small neighbourhood of $p$, which implies the global result by summation. 
\end{proof} 

\subsection{The modular orbifold}  

The orbifold $Y(1)=[\Gamma(1) \backslash \Hh]$ is the stacky quotient surface $\Gamma(1) \backslash \Hh$, with exactly two orbifold points: the orbit $[i]$ of $i$ with order $m=2$, and the orbit $[\rho]$ of $\rho = e^{2 \pi i/3}$ with order $m=3$. The local orbifold charts are depicted, e.g., in \cite[p.~7]{Gunning}. For example, the part in $\F$ of a small disk around $i$ in $\Hh$ is mapped first to the upper part of a disk $\widetilde D$ around $0$ by $z \mapsto (z-i)/(z+i)$, and then $\widetilde D$ is mapped $2$-to-$1$ to a disk $D$ by squaring ($z=w^2$).

Suppose $k=2 \ell$ is even and $f$ is a modular form for $\Gamma(1)$ not vanishing at the cusp; then $f(z) \d z^{\otimes \ell}$ is a section of line bundle $\Omega_{\Hh}^{\otimes \ell}$ on $\Hh$. We descend  $\Omega_{\Hh}$ to an orbifold line bundle $\mathcal K$ on $Y(1)$ (since we did not include a growth condition at infinity, tensor powers of this bundle maybe have an infinite-dimensional space of sections, but this is not relevant for our analysis below, that is entirely local). 

Write $x=\mathrm{Re}(z)$ and $y=\mathrm{Im}(z)$. On $\mathcal K^{\otimes \ell}$ we put the hermitian metric induced by the $\Gamma(1)$-invariant Petersson norm $||f(z) \d z^{\otimes \ell}||_h \defeq |f(z)|y^\ell$ (also here, the global integrability need not concern us, since we use the metric locally). The local weight is $-2 \ell \log y$ and the Chern form is $\ell/2\pi \cdot \d x \d y/y^2$, the hyperbolic area form up to a constant. Applying the orbifold Poincar\'e--Lelong formula and cancelling the $y^\ell$-term, we find the following. 
\begin{lemma} \label{PLmodform}  If $f$ is modular form of even weight $k$, not vanishing at the cusp, then we have an equality of currents \[ \mathrm{div}(f) = \frac{\Delta}{2 \pi} \log |f|. \]
on the modular orbifold $Y(1)$. \hfill \qedsymbol
\end{lemma}
To clarify  the formulation of the lemma, using directly the modular form $f$ instead of the corresponding section, note that although $\log |f|$ is not itself $\Gamma(1)$-invariant, $\Delta \log |f|$ is: for any $\gamma = \left( \begin{smallmatrix}  a & b \\ c & d \end{smallmatrix} \right) \in \Gamma(1)$, we have that the difference $\log |f(\gamma \cdot z)| - \log|f(z)| = k \log |cz+d|$ is harmonic, hence has vanishing Laplacian.

\section{Auxiliary results on the Hausdorff limit} \label{notaux} 

We recall the following notations and results from \cite{CCR}. 

\subsection{} \label{zerosincompact} Fix a congruence group $\Gamma$. \emph{For all $k$, $|Z_{k,\Gamma}|$ is contained in a compact domain $\DD_\Gamma \subset \Hh$,} see \cite[Thm.~A]{CCR}. A suitable $\DD_\Gamma$ can be determined precisely in terms of vanishing of so-called `Kluyver sums', which need not concern us here, but allows to make the results entirely concrete for any given congruence group. 

 Fix a sufficiently small connected open domain $\DD$ in $\Cc$ containing $\DD_\Gamma$. Note that $\DD$ may contain several $\Gamma(1)$-equivalent points; for example, $\DD$ will always contain points on the `right lower' arc $-\overline{\mathcal A}$ that is $\Gamma(1)$-equivalent to $\mathcal A$. it will later be important to eliminate the overcounting of zeros that this causes. 
  
 \subsection{} \label{hausdorfflimit} Fix a label $(a,b) \in (N^{-1}\Z/\Z)^2$ and $\gamma \in N_\Gamma$, set $F_k(z) \defeq E_{k,\gamma^{-1} \Gamma \gamma}^{a,b}$ regarded as a function in $\DD$, and rewrite this function in the form 
\[
F_k(z) = \sum_{\mathbf{v} \in \hat{\mathbb I}} \frac{c_{\mathbf v}}{f_{\mathbf v}(z)^k},
\]
where $f_{\mathbf v}(z) = mz + n$ is a family of linear polynomials, 
indexed by \[\hat{\mathbb{I}} \defeq \{\mathbf v =(m,n) \in (\Z_{> 0} \times \Z) \cup (\{ 0 \} \times \Z_{>0}) \colon c_{\mathbf v} \neq 0 \}. \]
This choice of indexing is important later; for example, since $k$ is even, if two non-zero terms $c_1/(mz+n)^k$ and $c_2/(-mz-n)^k$ for $m>0$ occur in the definition of the Eisenstein series, in the expression $F_k$ they are grouped into one term $(c_1+c_2)/f_{\mathbf v}(z)^k$ with $\mathbf v = (m,n)$, and $\mathbf v$ and $-\mathbf v$ do not both occur in $\hat{\mathbb{I}}$. 

 Write $\mathcal Z_k$ for the zero set of $F_k$ in $\DD$. Then there exists a finite subset $I \subset \hat{\mathbb{I}}$ such that the Hausdorff limit of these sets is 
\begin{align}  &\LH\limits_{\substack{k \rightarrow + \infty\\ k \in 2\Z_{\geq 2}}} \mathcal Z_k = Y \coloneqq \bigcup_{\substack{\mathbf{v}, \mathbf{w} \in I \\ \mathbf v \neq \mathbf w}}Y_{\mathbf v, \mathbf w}\nonumber \\ & \quad \text{with}\ \ Y_{\mathbf v,\mathbf w}  = \{ z \in \DD \colon |f_{\mathbf v}(z)| =|f_{\mathbf w}(z)| \leq  |f_{\mathbf v'}(z)| \mbox{ for all }\mathbf v' \in I  \}. \label{defY} \end{align} See \cite[Lemma~7.3.4]{CCR}. The finite subset $I$ can be determined explicitly in any given case, but this will not be needed in our work; see, however Subsection \ref{Phi} below.   

\subsection{} \label{geodesics} \emph{Non-empty solution sets to equations of the form $|f_{\mathbf v}(z)|=|f_{\mathbf w}(z)|$ in the complex upper half plane $\Hh$ for $\mathbf v=(m_1,n_1) \neq \mathbf \pm \mathbf w =  (m_2,n_2)$ in $\hat{\mathbb{I}}$ are geodesics.} More precisely, we can rewrite the defining equation as $|\sigma z|=1$ for $\sigma = \left(\begin{smallmatrix} m_1 & n_1 \\ m_2 & n_2 \end{smallmatrix} \right) \in \mathrm{M}_2(\Z) \cap \mathrm{GL}_2(\Q)$,  so the solution set is the image under $\sigma^{-1}$ of the unit circle.  In particular, $Y$ is contained in a finite number of \emph{geodesic segments}, meaning closed connected non-singleton subsets of a hyperbolic geodesic in the upper half plane. See \cite[Rem.~7.3.3]{CCR}.

\subsection{} \label{subharm} All $F_k$ are analytic on $\DD$; hence  \[ H_k(z) \defeq \frac{1}{k} \log |F_k(z)| \] 
is a well-defined function $H_k \colon \DD \rightarrow \R \cup \{-\infty\}$, continuous subharmonic on $\DD$ \cite[Thm.~2.2.2]{Rans} and taking the `value' $- \infty$ on $\mathcal Z_k$. Furthermore, $H_k$ is uniformly upper bounded on $\DD$. See \cite[Lemma~7.3.7]{CCR}. 

\subsection{}  \label{Phi} Define the function  $\Phi \colon \DD \rightarrow \R$ by 
\[ \Phi(z) = \max_{\mathbf v \in \hat{\mathbb{I}}} \{ - \log |f_{\mathbf v} (z)| \}. \]  
The following is a more convenient description of $\Phi(z)$. 
\begin{lemma} There exists a finite set $I$, depending only on $\Gamma$, such that 
\begin{equation} \label{Phifinite} \Phi(z) = \max_{\mathbf v \in I} \{ - \log |f_{\mathbf v} (z)| \}. \end{equation} 
\end{lemma}
\begin{proof} Suppose $\DD$ belongs to a box \[\DD\subseteq \{ z \colon (\mathrm{Re}(z), \mathrm{Im}(z)) \in [-M,M] \times [y_1,y_2] \}\] for some $M,y_1,y_2 \in \R_{>0}$. The proof of \cite[Lemma 7.3.4]{CCR} shows that $(0,C) \in \hat{\mathbb{I}}$ for some positive integer $C$. It follows easily that if $z=x+iy \in \DD$ and $|f_{(m,n)}(z)| = |mz+n| \leq |f_{(0,C)}| = C$, then $|m| y_1 \leq |m| y = |\mathrm{Im}(mz+n)| \leq C$ and $|n| \leq |m x + n| + |m x| \leq C + M|m|$. Thus, choosing 
\[ I \defeq \{ (m,n) \in \hat{\mathbb{I}} \colon |m| \leq C/y_1 \mbox{ and } |n| \leq C(1+M/y_1)\}; \]
we have that for every $\mathbf v = (m,n) \notin I$, $|f_{\mathbf v}(z)| > C = |f_{(0,C)}(z)|$ for every $z \in \DD$, with $(0,C) \in I$. 
\end{proof}

Now $\Phi$ is subharmonic by \cite[Thm.~2.2.3]{Rans}. Then 
\emph{$H_k(z)$ converges to $\Phi(z)$ uniformly in compact subsets of $\DD \setminus Y$, in particular, in} $L^1_{\text{loc}}(\DD \setminus Y)$, see \cite[Lemma~7.3.8]{CCR}. 

\section{Potential theory and measure convergence} \label{ptmc}

In this section, we find an expression of the limit of the orbifold counting measure in terms of functions of the type of the function $\Phi$ from \eqref{Phifinite}. We will use (more) potential theory in the complex plane, and refer to \cite{Rans} for the basic results on that topic. 

\subsection{} \label{extend} By Subsection \ref{subharm}, $H_k$ is subharmonic and uniformly upper bounded on $\DD$. Furthermore, no subsequence of $\{H_k\}$ converges locally uniformly to $-\infty$, since in any connected domain we can choose $z_0 \notin Y$, with $H_k(z_0)$ converging to $\Phi(z_0)>-\infty$. Now the compactness theorem for subharmonic functions \cite[Thm.~3.2.12]{HormConv} implies that every subsequence of $\{H_k\}$ has a subsequence converging to some subharmonic function $\Theta$ on $\DD$. Since $Y$, as union of finitely many geodesic segments, has Lebesgue measure zero, and the convergence of $H_k$ to $\Phi$ is locally uniform on the complement of $Y$, we find that $\Phi = \Theta$ almost everywhere on $\DD$, and then by the Weak Identity Principle for subharmonic functions \cite[Thm.~2.7.5]{Rans}, we find that $\Phi=\Theta$ on all of $\DD$, for any such $\Theta$ arising from the converging subsequence. We conclude that 
\begin{equation} \label{conv} H_k \rightarrow \Phi \mbox{ in }L^1_{\text{loc}}(\DD). \end{equation} 

\subsection{} The functions $F_k$ are not identically zero, and the functions $\log |F_k|$ are in $L^1_{\text{loc}}(\DD)$. Therefore, we can consider these functions as distributions in $\mathcal D'(\DD)$ (see, e.g., \cite[\S 6.11]{Rudin}) and we can apply the distributional Laplacian $\Delta = 4 \partial_z \partial_{\bar z}$ to them (viz., $\Delta f$ is defined by the identity $\langle \Delta f, \varphi \rangle = \langle f, \Delta \varphi\rangle$ for any test function $\varphi$). Applying this to both sides of \eqref{conv}, we conclude via \cite[Thm.~6.17]{Rudin} that 
\begin{equation} \label{convdist} \frac{\Delta}{2\pi} H_k \rightarrow  \frac{\Delta}{2\pi}  \Phi \mbox{ as distributions in }\mathcal D'(\DD). \end{equation} 
Since the functions $H_k$ and $\Phi$ are subharmonic, their Laplacians are positive distributions on $\DD$ \cite[Thm.~2.4.4]{Rans}, so they correspond to positive Radon measures in the Riesz(--Markov--Kakutani) representation theorem \cite[Thm.~A.3.2]{Rans}, and the convergence as distributions implies vague convergence of measures, i.e., convergence of any integral of a continuous compactly supported function (one may pass from convergence against smooth compactly supported test functions to continuous compactly supported test functions by uniform approximation on any fixed larger compact set). 

\subsection{} We now switch from fixing $a,b,\gamma$ and an individual Eisenstein series $F_k$ back to varying the data, and accordingly indicate the corresponding limit function as $\Phi_{\gamma^{-1} \Gamma \gamma}^{(a,b)}$. Looking at the product $\mathcal N_{k,\Gamma}$, by adding up \eqref{convdist} over all $a,b,\gamma$ we find 
\begin{equation} \label{fullconv}\lambda_k \defeq \frac{1}{k} \frac{\Delta}{2\pi} \log |\mathcal N_{k,\Gamma}|   \rightarrow \lambda \defeq \sum_{\substack{(a,b) \in (N^{-1}\Z/\Z)^2 \\ \gamma \in N_\Gamma}} \frac{\Delta}{2\pi} \Phi_{\gamma^{-1} \Gamma \gamma}^{(a,b)} \mbox{ in } \mathcal D'(\DD). \end{equation} 

Notice that the current $\lambda_k$ is invariant under the action of $\gamma = \left( \begin{smallmatrix}  a & b \\ c & d \end{smallmatrix} \right) \in \Gamma(1)$, since the difference $ \log |\mathcal N_{k,\Gamma}(\gamma \cdot z)| -  \log |\mathcal N_{k,\Gamma}(z)| = c_\Gamma \log | cz+d|$ is harmonic. Passing to the limit in \eqref{fullconv}, $\lambda$ inherits the same chart compatibility property: if $V \subseteq \DD$ is an open subset for which also $\gamma \cdot V \subseteq \DD$ for some $\gamma \in \Gamma(1)$, then for any compactly supported test function $\varphi$ on $\gamma \cdot V$, \[\langle \lambda|_{\gamma \cdot V}, \varphi \rangle = \lim\, \langle \lambda_k|_{\gamma \cdot V}, \varphi \rangle = \lim\, \langle \lambda_k|_V, \varphi \circ \gamma \rangle = \langle \lambda|_V, \varphi \circ \gamma \rangle.\]

\subsection{} We now wish to descend this result from $\DD$ to the modular orbifold. There are two things to take care of: (a) the fact that distinct points in $\DD$ can map to the same point on $Y(1)$, and the pushforward measure on $Y(1)$ should not double count these; (b) the fact that $Y(1)$ has orbifold points. Concerning point (a), if $D$ contains a small open disk $V_2$ in $\DD$ around a point on the right arc $-\overline{\mathcal A}$, the lower part does not belong to $\F$, and the disk is $\Gamma(1)$-equivalent to a disk $V_1$ centered at the corresponding point on the left arc $\mathcal A$. Now both of these disks together cover the corresponding open neighbourhood of the image (non-orbifold) point in $Y(1)$, whose measure in $Y(1)$ should then be $1/2$ the sum of the (equal) measures of the two disks in $\DD$, not merely their sum. See Figure \ref{orbequiv} for a corresponding picture. A similar analysis works for disks around points with real part $-1/2$ that have a $\Gamma(1)$-equivalent disk in $\DD$ around a point with real part $1/2$. Concerning point (b), as always, we should divide the measure by the order of the orbifold point. 

\begin{figure}[t]
 \centering
  \includegraphics[width=8cm]{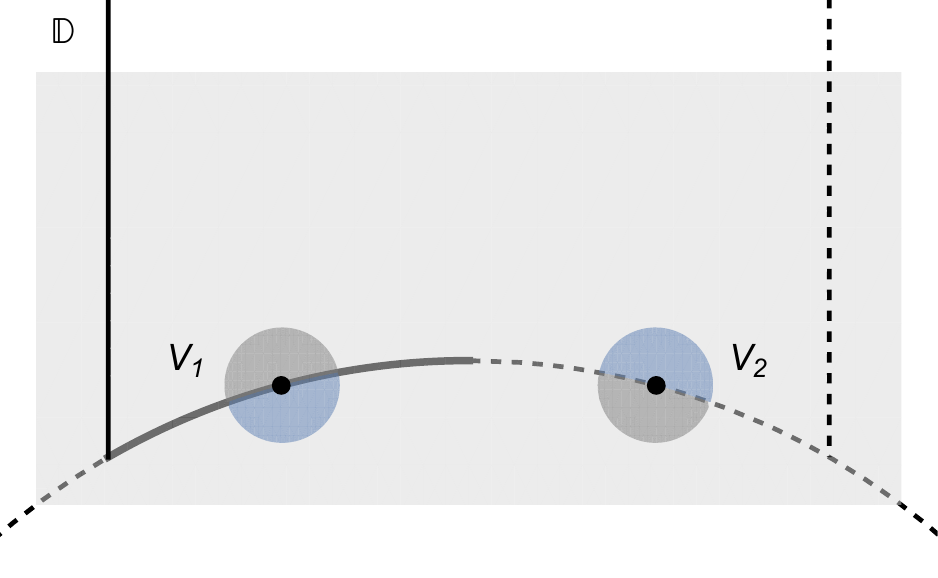}
\caption{A compact domain $\mathbb D$ enclosing the fundamental domain $\F$ and two $\Gamma(1)$-equivalent disks $V_1,V_2$ inside $\mathbb D$ centered on the boundary of $\overline \F$, with equivalent points shaded in the same grayscale.} \label{orbequiv}
  \end{figure}

Thus, we define the \emph{normalised descent map} $\widetilde \pi_*$ of a $\Gamma(1)$-invariant measure $\lambda$ from $\DD$ to $Y(1)$ at a point $p \in Y(1)$ with open neighbourhood $U$ by   
\[ (\widetilde \pi_* \lambda)|_U \defeq \frac{1}{m} \pi_*(\lambda|_{\tilde U}), \] 
if $\tilde U \rightarrow U$ is of degree $m$, where, for $p$ a non-orbifold point,  $\tilde U = \pi^{-1}(U) \subseteq \DD$ has $m$ sheets (e.g., in the example of Figure \ref{orbequiv}, $\tilde U = V_1 \sqcup V_2$ and $m=2$), and, for $p$ an orbifold point, $\tilde U \rightarrow U$ is the corresponding orbifold chart given by the action of an $m$-th root of unity (e.g., $m=2$ at $p=[i] \in Y(1)$). 

\begin{proposition} \label{propPhi} We have a weak convergence of measures 
\[ \mu_{k,\Gamma} \rightarrow \mu_{\infty, \Gamma} \defeq  \frac{6}{c_\Gamma \pi}  \widetilde \pi_* \hspace*{-5mm} \sum_{\substack{(a,b) \in (N^{-1}\Z/\Z)^2 \\ \gamma \in N_\Gamma}} \hspace*{-5mm} \Delta\Phi_{\gamma^{-1} \Gamma \gamma}^{(a,b)} \]
on the orbifold $Y(1)$, where $\widetilde \pi_*$ indicates the normalised descent map to the orbifold. 
\end{proposition} 

\begin{proof}  

For a test function $\varphi \in C_c^\infty(U)$ with $\tilde U \rightarrow U \subseteq Y(1)$ of degree $m$, 
\[\langle \widetilde \pi_* \lambda_k, \varphi \rangle = \frac{1}{m} \langle \lambda_k, \varphi \circ \pi \rangle \rightarrow \frac{1}{m} \langle \lambda, \varphi \circ \pi \rangle = \langle \widetilde \pi_* \lambda, \varphi \rangle.\] 

The measure $\widetilde \pi_* \lambda_k$ is amenable to application of Lemma \ref{PLmodform}, resulting, after multiplication with  $12/c_\Gamma$, in the indicated orbifold divisor $\frac{12}{kc_\Gamma} \mathrm{div}(\mathcal N_{k,\Gamma})$ that indeed corresponds,  by the valency formula \eqref{valence}, to the measure $\mu_{k,\Gamma}$ from the introduction. A zero in $\DD$ of order $r$ contributes a weight $12r/(c_\Gamma k m)$ to $\mu_{k,\Gamma}$. At a non-orbifold point with $m$ preimages, the $m$ equivalent lifts cancel the division by $m$, and at orbifold points, this is adjusted by the order of the stabiliser. 

Because the convergence before normalised descent is in $\mathcal D'(\DD)$, we arrive at vague convergence of measures on $Y(1)$. But the measures $\mu_{k,\Gamma}$ and their distributional limit $\mu_{\infty, \Gamma}$ are supported on a common compact subset of $Y(1)$ (enclosing the finitely many geodesic segments occurring in the Hausdorff limit), showing the tightness of the collection $\{\mu_{k,\Gamma}\}$. Hence the vague convergence in $\mathcal D'(\DD)$ implies weak convergence of measures on $Y(1)$, since `tightness plus vague convergence' implies weak convergence, see, e.g., \cite[Lemma 6.21]{Kallenberg}. 
\end{proof} 

\section{Identifying the limit measure locally on a geodesic segment in terms of Haar measure} \label{lmhaar}

We now identify the function on the right-hand side of the equation in Proposition \ref{propPhi} concretely in terms of hyperbolic geometry, using the bijection between $Y(1)$ and $\F$, in particular, between geodesics. For this, we switch back to the local picture, where $a,b,\gamma$ is fixed and left out of the notation. The support of the limiting measure for $F_k$ is the union of geodesic segments $Y_{\mathbf v,\mathbf w}$ of the form given in \eqref{defY}. As in the introduction, each such segment $E_\sigma$ is contained in the geodesic
$
\{z\in\mathbb H\mid |\sigma\cdot z|=1\},$
where $\sigma\in \mathrm{M}_2(\Z)\cap\mathrm{GL}_2(\Q)$ is the matrix constructed from $\mathbf v,\mathbf w$ in the notation of Subsection \ref{geodesics}. The segment $E_\sigma$ corresponds under $\sigma$ to a connected segment $I_\sigma$ of the complex unit circle.

\subsection{} We base this on a standard technique for analysing the distributional Laplacian of a jump function along a curve (see, e.g., \cite[Thm.~7.51]{Appel}); rather than specialising the general result to our special situation, we deduce it in a few lines from the Green--Stokes theorem. For simplicity, we fix $\sigma$ throughout and suppress the index from $E_\sigma$ (i.e. $E=E_\sigma$). Consider the function \[h(z) = -\log |f_{\mathbf v}(z)| + \log |f_{\mathbf w}(z)| = - \log |\sigma \cdot z|\] 
in an open neighbourhood $U$ of the interior $\mathring{E}$ of $E$, so that in $U$, $\mathring{E}$ is the zero locus of $h$, which divides $U\setminus \mathring{E}$ into $U^+ \defeq \{ z \in U \colon h(z)>0\}$ and $U^- \defeq \{ z \in U \colon h(z)<0\}$. Notice that on $U$, 
\begin{equation} \label{starPhi} \Phi(z) = - \log |f_{\mathbf{w}}(z)| + \frac{1}{2}(h + |h|). \end{equation} 
Indeed, the maximum in the definition of $\Phi$ can be realised by more than two indices $\mathbf v$ at only finitely many points. This is due to the fact that the $\{f_{\mathbf v}\}$ are linear forms, meaning that a point $z_0$ where  $|f_{\mathbf v'}(z_0)|=|f_{\mathbf v}(z_0)|=|f_{\mathbf w}(z_0)|$ for a third index $\mathbf v' \notin \{\pm \mathbf v, \pm \mathbf w\}$ belongs to two distinct geodesics described by $|\sigma\cdot z| = 1$ and $|\widetilde \sigma \cdot z|=1$ for two M\"obius transformations $\sigma, \widetilde \sigma$, and these intersect in only finitely many points. Furthermore, the intersection is transversal in the upper half plane. Hence along $E$, the quantity $|f_{\mathbf v'}|-|f_{\mathbf w}|$ changes sign at a triple point, so $f_{\mathbf v'}$ has smaller modulus on one side, and that side cannot belong to $Y_{\mathbf v, \mathbf w}$.  Hence triple points are boundary points of $E$.  

Since $ \log |f_{\mathbf{w}}(z)| $ and $h(z)$ are harmonic on $U$, so have vanishing Laplacian, we conclude from \eqref{starPhi} that 
$\Delta \Phi = \frac{1}{2} \Delta |h|$, which we can compute as a distribution using the Green-Stokes theorem, as follows: let $\varphi \in C_c^\infty(U)$ be a test function; then 
\[ \langle \Delta |h|, \varphi \rangle = \int_{U^+} h \Delta \varphi - \int_{U^-} h \Delta \varphi, \]
and, since $\Delta h = 0$,  
\[ \int_{U^+} h \Delta \varphi = \int_{U^+} h \Delta \varphi - \varphi \Delta h = \int_{\partial U^+} h \partial_n \varphi - \varphi \partial_n h, \]
where $\partial_n$ denotes the outer normal derivative. 
Since $\varphi$ is a test function with compact support the function itself as well as all its derivatives vanish in a neighbourhood of $\partial U^+ \setminus E$; so, in particular, it has outer normal derivative zero there. Furthermore, $h$ is zero on $E$, so what remains is 
\[ \int_{U^+} h \Delta \varphi = - \int_{E}  \varphi \partial_n h; \mbox{ and, similarly, }\int_{U^-} h \Delta \varphi = \int_{E}  \varphi \partial_n h, \]
and hence
\begin{equation} \label{ringE} \Delta \Phi = \frac{1}{2} \Delta |h| = - \partial_n h  = \partial_n \log |\sigma \cdot z| \mbox{ along } \mathring{E}. \end{equation} 

\subsection{} \label{noatoms} Since we have established \eqref{ringE} only on the interior of $E$, we need to consider the finitely many boundary points of $E$. We claim that the measure $\mu \defeq \Delta \Phi$ does not carry any atoms, which suffices for the purpose of extending the identity \eqref{ringE} across $E$: indeed, the identity of measures is given by integrating over $E$ against a test function, and without atoms, the integral over $E\setminus \mathring{E}$ is zero for $\mu$, and it is obviously zero for the argument measure.

 To prove the absence of atoms, note that $\Phi$ is Lipschitz on any compact in $\Hh$ since any $-\log |f_{\mathbf v}|$ is smooth with bounded gradient; say the Lipschitz constant is $L$. Choose a smooth, non-negative, radially symmetric compactly supported function $\chi$ on $B(0,2)$ that is identically $1$ on $B(0,1)$. 
For sufficiently small closed ball $\B \defeq B(p,r)$ of radius $r$ around a point $p \in D$ set $\chi_\B(z) = \chi((z-p)/r)$, so $||\Delta \chi_\B||_{L^1(D)} = ||\Delta \chi||_{L^1(\Cc)}$. Now $\B$ will have non-negative measure $\mu(\B) \leq \langle \Delta \Phi, \chi_\B \rangle$. With $\int \Delta \chi_\B =0$ by the divergence theorem, we find \[\mu(\B) \leq \int \Phi \Delta \chi_\B = \int (\Phi-\Phi(p)) \Delta \chi_\B \leq 2 L r  ||\Delta \chi||_{L^1(\Cc)},\] which tends to zero as the radius $r$ tends to zero. Hence by uppercontinuity of the measure $\mu$, we find $\mu(\{p\})=0$ for all points, and $\Delta \Phi$ has no atoms. 

\subsection{} The Cauchy--Riemann equations for the holomorphic function $\log \sigma \cdot z = \log |\sigma \cdot z| + i \arg(\sigma \cdot z)$ in $D$ give $\partial_n \log |\sigma\cdot z| = \partial_t \arg(\sigma\cdot z)$, where $\partial_t$ denotes the tangential derivative, so that finally, we find the following. 
\begin{proposition} \label{haarprop} Assume that $E_\sigma$ is a geodesic segment occurring in the limit zero set of an Eisenstein series for $\Gamma$, with corresponding limit function $\Phi$. Suppose that the corresponding geodesic has equation $|\sigma \cdot z|=1$ for some M\"obius transformation $\sigma \in  \mathrm{M}_2(\Z) \cap \mathrm{GL}_2(\Q)$, and $E_\sigma$ corresponds via $\sigma$ to a segment $I_\sigma$ on the unit circle. Then  
\begin{equation} \label{haar} \frac{1}{2\pi} \Delta \Phi |_{E_\sigma} = \frac{1}{2\pi} |\d \arg(\sigma \cdot z)|_{E_\sigma} = \sigma^* \frac{\d\theta}{2\pi}|_{E_\sigma} = \sigma^* \left(\mu_{S^1}|_{I_\sigma}\right), \end{equation}
which is the pullback of the normalised Haar measure from the unit circle restricted to the interval $I_\sigma$. \hfill \qedsymbol
\end{proposition} 

\begin{proof}[Proof of Main Theorem] 
The main theorem follows by using Proposition \ref{haarprop} for every term occurring in the right hand side sum of Proposition \ref{propPhi}. This is a finite sum of measures without atoms and these properties descend to $Y(1)$. 
\end{proof} 

\begin{remark} Equation \eqref{haar} is particularly easy to evaluate. If $z_1,z_2 \in E$, denote by $E[z_1,z_2]$ the unoriented subsegment of $E$ between the points $z_1$ and $z_2$. Then 
\[\frac{\Delta}{2\pi} \Phi (E[z_1,z_2])= \frac{1}{2\pi} \Big| \arg\Big(\frac{\sigma\cdot z_2}{\sigma \cdot z_1} \Big) \Big| = \frac{1}{2\pi} \arccos\Big(\Re\Big(\frac{\sigma \cdot z_2}{\sigma \cdot z_1} \Big)\Big),\]
where the result is in $[0,\pi]$. 
\end{remark} 

In the following examples, we will identify $Y(1)$ with $\F$, as before, and draw the geodesic segment in the limit inside $\F$. 

\begin{example} \label{gamman} We consider the example of a principal congruence group $\Gamma = \Gamma(N)$ for $N$ not divisible by $4$, in which case the Hausdorff limit was computed explicitly in \cite[\S 7.4]{CCR}. The Hausdorff limit of $|Z_{k,\Gamma(N)}|$ corresponds to the lower arc of the standard fundamental domain, $c_{\Gamma(N)}=N^2$, and our main result proves that the orbifold zero counting measure on  $Y(1)$ converges to 12 times the restriction of the Haar measure $|\d\theta|/2\pi$ to that segment of the unit circle: 
 \[\boxed{ \mu_{k,\Gamma(N)} \rightarrow \frac{6}{\pi} |\d\theta|_{\theta \in [\pi/2,2 \pi/3]}  \ \ (4 \nmid N).} \]
The segment having precisely arc length $\pi/6$, the limit is indeed a probability measure. 
\end{example} 

\begin{example} \label{gamman4} We consider the example of a principal congruence group $\Gamma = \Gamma(N)$ for $N$ divisible by $4$, in which case the Hausdorff limit was computed explicitly in \cite[\S 7.4]{CCR}. 
 Again, $c_{\Gamma(N)}=N^2$. In Figure \ref{arcs}, the $6$ distinct arcs $E_j$ ($j=1,\dots,6$) in the possible support of the limit measure in $\F$ are labelled and we also display a table in which one finds per arc how many of the $N^2$ choices of labels $(a,b) \in (N^{-1}\Z/\Z)^2$ result in limiting zeros along that $j$-th arc, what is the corresponding matrix $\sigma_j$, and what is the angular length (total measure times $2\pi$) of the corresponding arc $I_j$ on the unit circle, in terms of $\alpha=\arccos(1/4)$. For anything but arc 1 to occur, one needs $a=\pm 1/4$, and then, arc 2 and 3 occur together for $b \notin \{ \pm1/4,0,1/2\}$, arc 3 occurs without arc 2 for $b\in \{0,1/2\}$,  and arc 4, 5 and 6 together for $b=\pm 1/4$. Note that since several arcs occur for the same Eisenstein series, and some Eisenstein series have no zeros in the standard fundamental domain (that is, after all, not $\Gamma(N)$-invariant), the total sum of multiplicities is not $N^2$. Also notice that, apart from the case where $N=4$, all six arc occur in the support of the limit measure, whilst for $N=4$, all but arc 2 do so. 
 
Explicitly, writing $|\d\theta_j| \defeq |\d\, \mathrm{arg}(\sigma_j \cdot z)|_{E_j} = (\sigma_j|_{E_j})^* |d\theta|_{I_j}$ for measure concentrated on the $j$-th arc and equal to the angular length of $I_j$ there, and setting \[|\d\theta_{456}| = |d\theta_4| +  |\d\theta_5| + |\d\theta_6|,\] we find as limit of the zero counting measure
\[\boxed{\mu_{k,\Gamma(N)} \rightarrow \frac{6}{\pi N^2} \Bigg( (N-2)^2 |\d\theta_1|+ (2N-8) |\d\theta_2| + (2N-4) |\d\theta_3|+ 4 |d\theta_{456}| \Bigg) \ \ (4 \mid N).} \]
  As a sanity check, replacing in the right hand side expression each $|\d \theta_j|$ by the corresponding total measure of the corresponding arc, one finds $1$, so that the limit is indeed a probability measure. 
\begin{figure}[t]
  \begin{minipage}[c]{0.29\textwidth}
    \centering
 \centering
  \includegraphics[width=4cm]{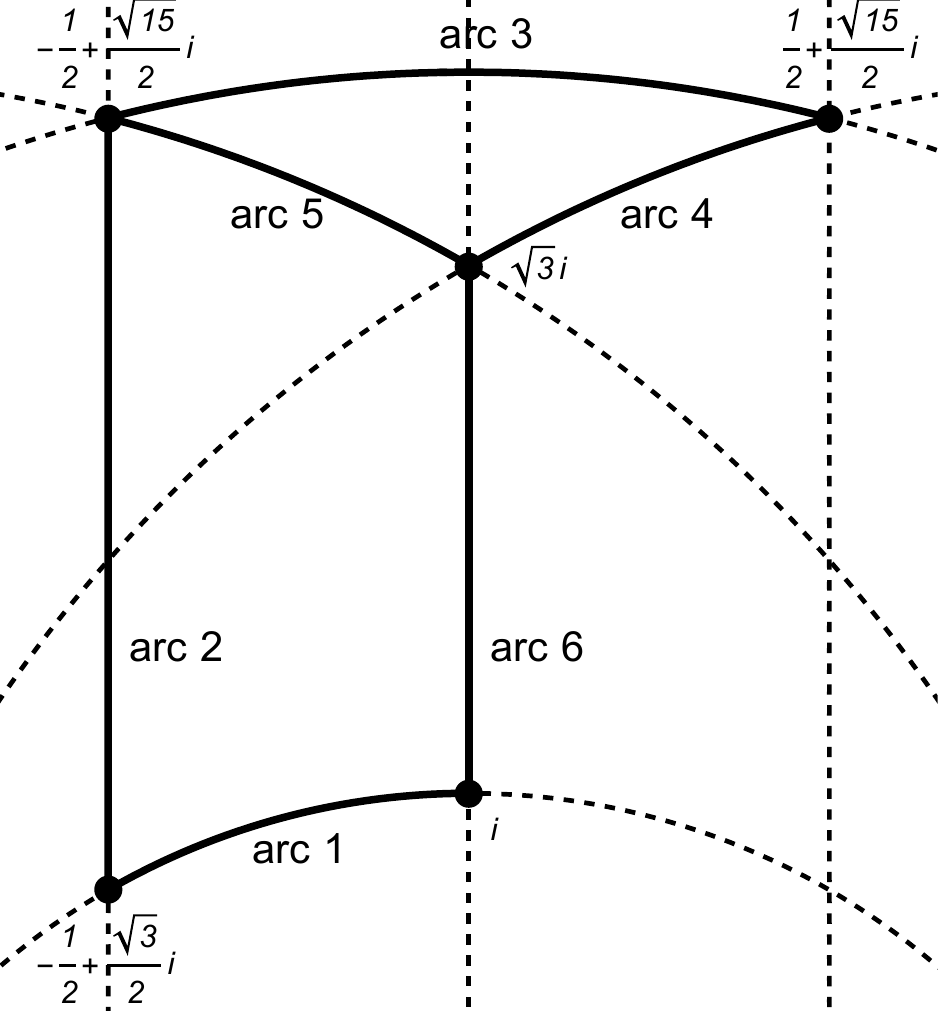}
  \end{minipage}
  \hfill
  \begin{minipage}[c]{0.69\textwidth}
    \centering
 {\tiny  
 \bgroup
\def\arraystretch{1.3}
  \begin{tabular}{l|ccccccc}
 arc $j$ & arc 1 & arc 2 & arc 3 & arc 4 & arc 5 & arc 6 \\
\hline
multiplicity & $(N-2)^2$ & $2N-8$ & $2N-4$ & $4$ & $4$ & $4$ \\
$\sigma_j \cdot z$ & $z$ & $\frac{z+1}{z}$ & $\frac12 z$ & $\frac12 (z-1)$ & $\frac12 (z+1)$ & $\frac{z+1}{z-1}$ \\ 
angular length fo $I_j$ & $\frac{\pi}{6}$ & $2(\alpha-\frac{\pi}{3})$ & $\pi-2\alpha$ & $\alpha-\frac{\pi}{3}$ & $\alpha-\frac{\pi}{3}$ & $\frac{\pi}{6}$ \\[2mm]
\end{tabular}
\egroup
}
\end{minipage}
\caption{Left: six numbered geodesic segments (arcs; denoted by solid lines) in the standard fundamental domain for $\Gamma(1)$; Right: for $N$ divisible by $4$, a table with, per arc, how many of the $N^2$ Eisenstein series have zeros tending to it, what matrix $\sigma$ maps it to the unit circle, and what is the angular length of the corresponding segment of the unit disk (integral of $\Delta \Phi$ over the entire arc), in terms of $\alpha=\arccos(1/4)$. Notice that for $N=4$, arc 2 does not occur, but all other five do; otherwise, all six distinct arcs occur in the support of the limit.} \label{arcs}
  \end{figure}
%
\end{example} 

\subsection{} \label{finalsec} The nonatomicity of the limit measure can be used to prove the following two corollaries. 

\begin{corollary} \label{smallo} 
Fix a group $\Gamma$. Every compact subset disjoint from the Hausdorff limit  of $|Z_{k,\Gamma}|$ is eventually free of zeros. Furthermore, 
\[ \sup_{x \in Y(1)} \mathrm{ord}_x(\mathcal N_{k,\Gamma}) = o(k).\] 
\end{corollary} 

\begin{proof} 
The first statement is clear from the Hausdorff convergence. 
For the second, first notice that \begin{equation} \label{supmu} \sup_{x \in Y(1)} \mu_{k,\Gamma}(\{x\}) \rightarrow 0.\end{equation} If this were not true, there would be a sequence of points $\{x_n\}$ and weights $\{k_n\}$ with $\mu_{k_n,\Gamma}(\{x_n\})>\varepsilon$ for some absolute $\varepsilon>0$. Since all zeros belong to a fixed compact set (the image of $\DD_\Gamma$ in $Y(1)$)  independent of the weight, this sequence then has a further subsequence, by abuse of notation also denoted $\{x_n\}$, converging to $x$ in $Y(1)$. Eventually, $x_n$ enters the closed disk $B_r$ of radius $r$ around $x$, so 
\[ \varepsilon \leq \limsup_k \mu_{k,\Gamma}(B_r) \leq \mu_{\infty, \Gamma}(B_r), \] 
by the closed-set inequality from the Portmanteau theorem for weak convergence \cite[Theorem 5.25]{Kallenberg}.
Letting $r$ tend to zero, by upper continuity of the limit measure, this would imply that $\mu_{\infty, \Gamma}(\{x\}) \geq \epsilon$, contradicting the fact that the limit measure has no atoms, as was shown in Subsection \ref{noatoms}. This proves \eqref{supmu}. Since $\# \mathrm{Stab}_{\Gamma(1)}(x) \leq 3$, we find 
\[ \sup_{x \in Y(1)} \frac{1}{k} \cdot \mathrm{ord}_x(\mathcal N_{k,\Gamma}) \leq \frac{c_\Gamma}{4} \sup_{x \in Y(1)} \mu_{k,\Gamma}(\{x\}), \]
  The convergence to zero  of the right hand side implies the $o(k)$ claim for $\sup_{x} \mathrm{ord}_x(\mathcal N_{k,\Gamma})$. 
\end{proof}

Instead of the orbifold zero counting measure, consider its `unweighted' counterpart $\nu_{k,\Gamma}$ defined by 
\[ \nu_{k,\Gamma} \defeq \frac{1}{\#Z_{k,\Gamma}}  \sum_{x \in Y(1)} \mathrm{ord}_x(\mathcal N_{k,\Gamma}) \delta_x. \]
In this formula, $\#Z_{k,\Gamma} = \sum_{x \in Y(1)} \mathrm{ord}_x(\mathcal N_{k,\Gamma})$ is the number of elements of the \emph{multiset} of zeros, each counted with multiplicity. The difference with respect to the orbifold measure is that the elliptic points are not counted with fractional weight given by the order of their stabilizer. 

\begin{corollary} \label{noorb} 
Fix a group $\Gamma$. We have a weak convergence of measures  on $Y(1)$
\[ \nu_{k,\Gamma} \rightarrow \mu_{\infty, \Gamma}. \] 
\end{corollary} 

\begin{proof} We will show that the so-called total variation satisfies \begin{equation} \label{totvar} \sup_{B \subseteq Y(1)} |\nu_{k,\Gamma}(B) - \mu_{k,\Gamma}(B)| = o(1), \end{equation} where $B$ runs over measurable sets; it follows from \eqref{totvar} that the two sequences have the same weak limit \cite{GS}. 

To prove \eqref{totvar}, we use the shorthand notation 
$ a_k \defeq \mathrm{ord}_{[i]} (\mathcal N_{k,\Gamma})$, $b_k \defeq \mathrm{ord}_{[\rho]} (\mathcal N_{k,\Gamma})$ and $z_k \defeq \# Z_{k,\Gamma}$
in the above sense of multiset cardinality, so that $z_k = \deg(Z_{k,\Gamma}) + a_k/2+2b_k/3$. By Corollary \ref{smallo}, $a_k= o(k)$ and $b_k= o(k)$, and hence also $z_k = \deg(Z_{k,\Gamma}) + o(k) = c_\Gamma k/12 + o(k)$. This allows us to conclude that both $a_k/z_k$ and $b_k/z_k$ are $o(1)$.  
We have 
\[ \nu_{k,\Gamma} = \frac{1}{z_k} \Big(\deg Z_{k,\Gamma} \cdot \mu_{k,\Gamma} + \frac{a_k}{2} \delta_{[i]} + \frac{2b_k}{3} \delta_{[\rho]}\Big), \]
and hence for any measurable set $B \subseteq Y(1)$, 
\[ |\nu_{k,\Gamma}(B) - \mu_{k,\Gamma}(B)| \leq \Big|\frac{c_\Gamma k}{12 z_k} - 1\Big| \cdot \mu_{k,\Gamma}(B) +  \frac12\frac{a_k}{z_k} + \frac23\frac{b_k}{z_k} = o(1), \]
using the above estimates. 
\end{proof}

\bibliographystyle{amsplain}

\end{document}